%% file: main.tex
\documentclass[oneside,abstracton]{scrartcl}
\usepackage[utf8]{inputenc}
\usepackage[T1]{fontenc}
\usepackage{lmodern}
\usepackage[english]{babel}

\usepackage{amsmath}
\usepackage{amsthm}
\usepackage{amssymb}
\usepackage[nobysame,initials]{amsrefs}
\usepackage{enumerate}
\usepackage{cite}
\usepackage{mleftright}

\swapnumbers
\mleftright

\theoremstyle{plain}
	\newtheorem{theorem}{Theorem}[section]
	\newtheorem{lemma}[theorem]{Lemma}
	\newtheorem{corollary}[theorem]{Corollary}

\newcommand{\cK}	{\mathcal K}
\newcommand{\cP}	{\mathcal P}
\newcommand{\cQ}	{\mathcal Q}
\newcommand{\cR}	{\mathcal R}
\newcommand{\GLn}[1][n]	{\operatorname{GL}(#1)}
\newcommand{\Id}	{\operatorname{Id}}
\newcommand{\muTr}	{\tilde \mu}

\newcommand{\ptwo}[2]	{\begin{pmatrix} #1 \\ #2 \\ \end{pmatrix}}
\newcommand{\ptwotwo}[4]{\begin{pmatrix} #1 & #3 \\ #2 & #4 \\ \end{pmatrix}}

\newcommand{\R}		{\mathbb R}
\newcommand{\spn}	{\operatorname{span}}
\newcommand{\SLn}[1][n]	{\operatorname{SL}(#1)}

\newcommand{\VLn}[1][n]	{\operatorname{SL}^\pm (#1)}

\title{Moments and Valuations}
\author{Christoph Haberl and Lukas Parapatits}
\date{}

\begin{document}
	\maketitle

	\begin{abstract}
		\noindent
		\input{abstract} \\[0.5cm]
	\end{abstract}

	\section{Introduction}\label{intro}
		\input{intro}

	\section{Notation and preliminary results}\label{prelim}
		\input{preliminaries}

	\section{Proof of the Main Results}
		\input{proof}

	\section{Acknowledgements}
		\input{acknowledgements}

	\goodbreak

	\vspace{1cm}
	\noindent
	Christoph Haberl \\
	Vienna University of Technology \\
	Institute of Discrete Mathematics and Geometry \\
	Wiedner Hauptstraße 8--10/104 \\
	1040 Wien, Austria \\
	christoph.haberl@tuwien.ac.at

	\vspace{1cm}
	\noindent
	Lukas Parapatits \\
	ETH Zurich\\
	Department of Mathematics \\
	Rämistrasse 101\\
	8092 Zürich, Switzerland \\
	lukas.parapatits@math.ethz.ch
\end{document}

%% file: abstract.tex
All measurable and $\SLn$-covariant vector valued valuations on convex polytopes
containing the origin in their interiors are completely classified.
The moment vector is shown to be essentially the only such valuation.

%% file: intro.tex
Functionals which are compatible with the geometric and topological structure of their underlying space
are of vital importance in geometry.
In convex geometric analysis, valuations have been studied from this perspective for decades.
Valuations are functionals $\mu \colon \mathcal{S}\to\langle A, + \rangle$ defined on a collection of sets $\mathcal{S}$
with values in an abelian semigroup $\langle A, + \rangle$ such that
\[
	\mu (K \cup L) + \mu (K \cap L) = \mu (K) + \mu (L)
\]
whenever $K$, $L$, $K\cup L$, $K\cap L$ are contained in $\mathcal{S}$.

Due to their critical role in Dehn's solution of Hilbert's Third Problem,
the interest in valuations dates back to the beginning of the twentieth century.
A systematic study of valuations was initiated later by Hadwiger.
This culminated in Hadwiger's celebrated characterization theorem,
where he classified all continuous and rigid motion invariant valuations on the space of convex bodies
(i.e.\ nonempty compact convex subsets of $\R^n$ equipped with Hausdorff distance).
Hadwiger's theorem shows that the vector space of such valuations is finite dimensional
and a basis is given by the intrinsic volumes.
The latter are generalizations of such basic notions as volume, surface area and mean width.

In other words, Hadwiger's result revealed that basic geometric functionals can be characterized as valuations
compatible with certain linear maps and the topology induced by the Hausdorff distance.
This way of looking at functionals in convex geometry turned out to be extremely fruitful.
Indeed, numerous geometric objects have been characterized in this way over the last years.
Examples include mixed volumes, affine surface areas, the projection body operator and the intersection body operator
(see e.g. \cite{Alesker99, bernig08, Klain96, Lud11, Lud10, Lud12, Hab11, Hab10amj, habpar13a, Schduke, Par10a, Parapatits2014,
SchWan10, AbaBer11, Lud06ajm, Lud03, Lud02advproj, Wannerer2011, wannerer13a}).

Let us mention two examples which illustrate that it pays off to characterize valuations in this way.
First, Alesker's ingenious classification \cite{Alesker01} of continuous and translation invariant valuations
not only solved the long-standing McMullen conjecture,
but also serves as the basis of a new algebraic integral geometry
(see e.g. \cite{Alesker03,alesker04,bernig07a,BerFu10,fu06,wannerer13,aleber12}).
Second, Ludwig's seminal work on body valued valuations \cite{Lud05}
paved the way for strengthenings of the sharp $L_p$ Sobolev, the Moser-Trudinger and the Morrey-Sobolev inequalities
(see e.g. \cite{CLYZ, HSX, LutYanZha02jdg}).

Let $\cK_o^n$ denote the set of convex bodies containing the origin in their interiors.
We write $\cP_o^n$ for the subset of $\cK_o^n$ consisting of polytopes only.
In order to obtain Hadwiger type theorems in centro-affine geometry,
it turned out that one has to consider valuations defined on $\cK_o^n$.
This restriction is necessitated by the evolution of the classical Brunn-Minkowski theory towards an Orlicz-Brunn-Minkowski theory.
During this process, several new operators have been discovered and investigated
(see e.g. \cite{BorLutYanZha, garhugwei13, HabLutYanZha09, LutYanZha09badv, LutYanZha09jdg, LR10, wer13, ye13, sta12}).
These new objects are far reaching generalizations of classical notions,
but in most cases they are defined only on $\cK_o^n$.
So aiming at characterizations of these new operators,
one has to describe valuations on $\cK_o^n$ and $\cP_o^n$, respectively.
As an example, the authors recently obtained the following Hadwiger type theorem \cite{habpar13}.
\begin{theorem}\label{th: main0}
	Let $n \geq 2$.
	A map $\mu \colon \cP_o^n\to\R$ is an upper semicontinuous and $\SLn$-invariant valuation
	if and only if there exist constants $k_0, k_1, k_2 \in \R$ such that
	\[
		\mu (P) = k_0 \, \chi (P) + k_1 \, V(P) + k_2 \, V(P^*)
	\]
	for all $P \in \cP_o^n$.
\end{theorem}
Here, $\chi$ is the Euler characteristic, $V$ denotes $n$-dimensional volume
and $P^*$ is the polar body of $P$ (see Section \ref{prelim} for details).

The impact of Theorem \ref{th: main0} to the Orlicz-Brunn-Minkowski theory is revealed if one combines
it with a deep result of Ludwig and Reitzner \cite{LR10} on affine surface areas.
In this way, the authors \cite{habpar13} obtained the following centro-affine Hadwiger theorem:
A map $\mu \colon \cK_o^n \to \R$ is an upper semicontinuous and $\SLn$-invariant valuation
if and only if there exist constants $k_0, k_1, k_2 \in \R$ and a function $\varphi \in \textnormal{Conc}(\R_+)$ such that
\[
	\mu (K) = k_0 \, \chi (K) + k_1 \, V(K) + k_2 \, V(K^*) + \Omega_{\varphi} (K)
\]
for all $K \in \cK_o^n$.
The Orlicz affine surface areas $\Omega_{\varphi}$ were discovered only recently
and we refer to \cite{habpar13} for a precise definition of these functionals and the set $\textnormal{Conc}(\R_+)$.

The aim of this article is to establish the analog of Theorem \ref{th: main0} for vector valued valuations.
In particular, we strengthen previous characterizations by Ludwig \cite{Ludwig:moment}.
Whereas in the scalar case the natural compatibility with the special linear group is given by $\SLn$-invariance,
the appropriate notion in the vector case is $\SLn$-covariance.
A map $\mu \colon \cP_o^n\to\R^n$ is called $\SLn$-covariant if
\begin{equation}\label{covariant}
	\mu (\phi P) = \phi \, \mu (P)
\end{equation}
for all $P \in \cP_o^n$ and each $\phi \in \SLn$.
Our main result is the following theorem.
\begin{theorem}\label{th: main1}
	Let $n \geq 3$.
	A map $\mu \colon \cP_o^n\to\R^n$ is a measurable and $\SLn$-covariant valuation
	if and only if there exists a constant $k \in \R$ such that
	\[
		\mu (P) = k \, m(P)
	\]
	for all $P \in \cP_o^n$.
\end{theorem}
The vector $m(P)$ in Theorem \ref{th: main1} is the moment vector of the polytope $P \in \cP_o^n$.
For each $P \in \cP_o^n$, it is defined by
\[
	m(P) = \int_P x \, dx .
\]
Up to volume normalization, the moment vector coincides with the center of gravity of $P$.
This makes it a basic notion in mechanics, engineering, physics and geometry.

In dimension two, the situation is different.
In contrast to Theorem \ref{th: main1},
the vector space of measurable and $\SLn[2]$-covariant valuations turns out to be two-dimensional.
Indeed, if we denote by $\rho_{\frac \pi 2}$ the counter-clockwise rotation of $\R^2$ about the angle $\frac \pi 2$,
then we will prove the following result.
\begin{theorem}\label{th: main1 2d}
A map $\mu \colon \cP_o^2 \to \R^2$ is a measurable and $\SLn[2]$-covariant valuation
if and only if there exist constants $k_1, k_2 \in \R$ such that
\[
	\mu (P) = k_1 \, m(P) + k_2 \, \rho_{\frac \pi 2} \, m (P^*)
\]
for all $P \in \cP_o^2$.
\end{theorem}

As mentioned before, Ludwig \cite{Ludwig:moment} was the first who obtained classifications in this centro-affine framework.
However, she assumed covariance with respect to the whole general linear group.
Theorems \ref{th: main1} and \ref{th: main1 2d} do not need homogeneity assumptions at all.
In fact, since the moment vector is homogeneous, Theorem \ref{th: main1} shows that $\SLn$-covariance implies homogeneity.
We remark that prior to Ludwig's work,
a Hadwiger type theorem for vector valued valuations was established by Schneider under different assumptions.
We refer to \cite{Sch13} and the references therein for more information on this subject.

The results of this article can be regarded as the first step
towards a complete classification of $\SLn$-covariant tensor valuations.
Such tensor valuations were recently investigated from different perspectives,
see e.g. \cite{AleBerSchu, Hugschsch08, Lud03, wannerer13, Ludwig:moment, hugsch13, berhu14}.
Maps $\mu \colon \cP_o^n \to (\R^n)^{\otimes p}$ that naturally intertwine the special linear group $\SLn$, i.e.\
\begin{equation}\label{covarianttensor}
	\mu \circ \phi = \phi^{\otimes p} \, \mu
\end{equation}
for all $\phi \in \SLn$, are called $\SLn$-covariant.
If one identifies $(\R^n)^{\otimes 1}$ with $\R^n$ in the trivial way,
then clearly the two notions (\ref{covariant}) and (\ref{covarianttensor}) of $\SLn$-covariance correspond to each other.
So Theorems \ref{th: main1} and \ref{th: main1 2d} classify
all measurable and $\SLn$-covariant valuations $\mu \colon \cP_o^n \to (\R^n)^{\otimes 1}$.

To prove our main results we will need a generalization of Theorem \ref{th: main0},
which is interesting in its own right.
We will prove in this article that the assumption on upper semicontinuity can be weakened.
In fact, in Section \ref{prelim} we establish the following theorem which shows that measurability is sufficient.

\begin{theorem}\label{th: main2}
	Let $n \geq 2$.
	A map $\mu \colon \cP_o^n\to\R$ is a measurable and $\SLn$-invariant valuation
	if and only if there exist constants $k_0, k_1, k_2 \in \R$ such that
	\[
		\mu (P) = k_0 \, \chi (P) + k_1 \, V(P) + k_2 \, V(P^*)
	\]
	for all $P \in \cP_o^n$.
\end{theorem}

%% file: preliminaries.tex
We equip $\R^n$ with the standard Euclidean product and fix an orthonormal basis.
The corresponding basis vectors are denoted by $e_1, \ldots, e_n$.

We will often identify $\R^{n-1}$ with $e_n^\perp$.
With respect to this identification, $a_*$ will denote the first $n-1$ coordinates of a vector $a \in \R^n$.
The subscript $\cdot_{\, *}$ will also be used to denote the first $n-1$ coordinates of other objects.
For example, suppose that a matrix $A \in \R^{n \times n}$ is given.
Then $A_{**} \in \R^{(n-1) \times (n-1)}$ is obtained from $A$ by deleting its $n$-th row and $n$-th column.
Similarly, $A_{*n} \in \R^{n-1}$ is obtained from the $n$-th column of $A$ by deleting its $n$-th coordinate.
A superscript $\cdot \, '$ will denote the $(n-1)$-dimensional version of an $n$-dimensional object.
For example, $V'$ will denote the $(n-1)$-dimensional volume.

When talking about a measurable function between topological spaces,
we always understand the notion of measurability with respect to their Borel $\sigma$-algebras.

The well-known solution of Cauchy's functional equation will be one of the main ingredients in our proofs.
Suppose that $f \colon \R^n \to \R$ is a measurable function such that
\begin{equation}\label{eq: Cauchy}
	f(s + t) = f(s) + f(t)
\end{equation}
for all $s,t \in \R^n$.
Then $f$ has to be linear.
If $f$ maps to $\R^m$ instead of $\R$, then the corresponding statement also holds,
which can be easily deduced from the case $m=1$.
In the case $n=1$, it is actually enough that \eqref{eq: Cauchy} holds for positive real numbers.

Let $\cK^n$ denote the set of convex bodies, i.e.\ nonempty compact convex subsets of $\R^n$.
As usual, $\cK^n$ will be equipped with the Hausdorff metric.
We write $\cP^n \subseteq \cK^n$ for the subspace of convex polytopes.
The space $\cP^{n-1}$ will be repeatedly identified with convex polytopes that are contained in $e_n^\perp$.

In order to keep formulas easy to read, we will use the following conventions.
The convex hull of $P_1 \cup \ldots \cup P_m$ will be denoted by $[P_1, \ldots, P_m]$, where $P_1, \ldots, P_m \in \cP^n$.
Whenever a set contains only one point, we will omit the curly brackets in the above notation.
In particular, $E_1, \ldots, E_n$ denote the line segments $[-e_1, e_1], \ldots, [-e_n, e_n]$.
If a map $\mu$ is applied to $[P_1, \ldots, P_m]$, then we will usually write $\mu [P_1, \ldots, P_m]$ 
instead of $\mu \left( [P_1, \ldots, P_m] \right)$.

In the introduction we already used the concept of polar bodies.
The polar body $K^* \in \cK^n_o$ of a convex body $K \in \cK^n_o$ is defined by
\[
	K^* = \{ y \in \R^n : y^t  x \leq 1 \text{ for all } x \in K \} .
\]
Recall that the map $K \mapsto K^*$ is a continuous involution on $\cK^n_o$.
Moreover, for bodies $K , L \in \cK^n_o$ such that $K \cup L$ is convex we have
\begin{equation}\label{valuation polar bodies}
	( K \cup L )^* = K^* \cap L^*
	\quad \textrm{ and } \quad
	( K \cap L )^* = K^* \cup L^* .
\end{equation}
If $\phi \in \GLn$, then
\begin{equation}\label{contravariance polarity}
	(\phi K )^* = \phi^{-t} K^* ,
\end{equation}
where $\phi^{-t}$ denotes the inverse of the transpose of $\phi$.
We refer the reader to \cite{Gar95} and \cite{Sch13} for proofs of these facts.

The following symbols will have a fixed meaning throughout this article.
The letters $a, b, c, d$ will always denote positive real numbers
with associated line segments $I := [-a e_1, b e_1]$ and $J := [-c e_n, d e_n]$, respectively.
The letters $x, y$ will always denote elements of $\R^{n-1}$.
In particular, for $n = 2$ we have $J = [-c e_2, d e_2]$ and $x,y \in \R$.
The letter $B$ will always denote an element of $\cP^{n-1}_o$.

For $n = 2$, we say that $a, b, c, d, x, y$ form a double pyramid if
\[
	\left[ I, -c \ptwo x 1, d \ptwo y 1 \right] \cap e_2^\perp = I .
\]
For $n \geq 3$, we say that $B, c, d, x, y$ form a double pyramid if
\[
	\left[ B, -c \ptwo x 1, d \ptwo y 1 \right] \cap e_n^\perp = B .
\]
If $x = y = 0$, then we call the double pyramid straight.
The set of double pyramids will be denoted by $\cR^n$
and the set of straight double pyramids by $\cQ^n$.
Using a slightly different notation, the next theorem was proved by Ludwig in \cite{Lud02advval}.

\begin{theorem}\label{th: determined by values on SL(n)(R^n)}
	Let $n \geq 2$.
	Assume that $\mu \colon \cP^n_o \to \R^n$ is a valuation which vanishes on all $\SLn$-images of elements in $\cR^n$.
	Then $\mu$ vanishes everywhere.
\end{theorem}

The notion of $\SLn$-covariance has already been introduced in the previous section.
However, we need additional terminology for functions which intertwine the special linear group.
A map $\mu \colon \cP^n_o \to \R^n$ is called $\SLn$-contravariant if
\[
	\mu (\phi P) = \phi^{-t} \, \mu (P)
\]
for all $P \in \cP^n_o$ and each $\phi \in \SLn$.
If the domain of $\mu$ is only a subset of $\cP^n_o$ and not necessarily closed under the action of $\SLn$,
we require the covariance respectively contravariance property to hold
only for those combinations of $P$ and $\phi$ that make sense.

The notions of $\SLn$-covariance and $\SLn$-contravariance are closely related to each other.
Indeed, from relation \eqref{contravariance polarity} we deduce the following.
If $P \mapsto \mu (P)$ is $\SLn$-covariant, then $P \mapsto \mu (P^*)$ is $\SLn$-contravariant.
Vice versa, if $P \mapsto \mu (P)$ is $\SLn$-contravariant, then $P \mapsto \mu (P^*)$ is $\SLn$-covariant.

The group of all volume preserving linear maps, i.e.\ those with determinant $1$ or $-1$, will be denoted by $\VLn$.
A map $\mu \colon \cP^n_o \to \R^n$ is called $\VLn$-covariant if
\begin{equation}\label{eq: definition VL(n)-covariant}
	\mu (\phi P) = \phi \, \mu (P)
\end{equation}
for all $P \in \cP^n_o$ and each $\phi \in \VLn$.
We say that a map is $\VLn$-signum-covariant if
\begin{equation}\label{eq: definition VL(n)-signum-covariant}
	\mu (\phi P) = (\det \phi) \, \phi \, \mu (P)
\end{equation}
for all $P \in \cP^n_o$ and each $\phi \in \VLn$.
Again, if the domain of $\mu$ is only a subset of $\cP^n_o$ and not necessarily closed under the action of $\VLn$,
we require \eqref{eq: definition VL(n)-covariant} respectively \eqref{eq: definition VL(n)-signum-covariant} to hold
only for those combinations of $P$ and $\phi$ that make sense.

Let us give one example for each of the last two concepts.
From the transformation behaviour of integrals with repect to linear maps,
it is easy to see that the moment vector $m$ is $\VLn$-covariant.
In the plane, a simple calculation shows that the map $P \mapsto \rho_{\frac \pi 2} m (P^*)$
is $\VLn[2]$-signum-covariant.

Assume that $\mu \colon \cP^n_o \to \R^n$ is a measurable valuation which is $\SLn$-covariant
and let $\theta \in \VLn \setminus \SLn$.
For all $P \in \cP^n_o$ define
\begin{equation}\label{eq: definition mu^+}
	\mu^+ (P) = \frac 1 2 \left( \mu (P) + \theta \, \mu (\theta^{-1} P) \right)
\end{equation}
and
\begin{equation}\label{eq: definition mu^-}
	\mu^- (P) = \frac 1 2 \left( \mu (P) - \theta \, \mu (\theta^{-1} P) \right) .
\end{equation}
The $\SLn$-covariance of $\mu$ implies that these definitions do not depend on the choice of $\theta$.
Clearly, $\mu^+$ and $\mu^-$ are measurable valuations.
Moreover, it is easy to see that $\mu^+$ is $\VLn$-covariant and $\mu^-$ is $\VLn$-signum-covariant.
Obviously,
\begin{equation}\label{eq: mu = mu^+ + mu^-}
	\mu = \mu^+ + \mu^- .
\end{equation}

In order to establish our main result in dimensions greater or equal than three, we need a generalization
of Ludwig's characterization \cite{Lud03} of matrix valued valuations.
Before we can formulate her theorem, we have to collect some more definitions.
A map $\mu \colon \cP^n_o \to \R^{n \times n}$ is called $\GLn$-covariant if
\begin{equation}
	\mu (\phi P) = |\det \phi| \, \phi \, \mu (P) \, \phi^t
\end{equation}
for all $P \in \cP^n_o$ and each $\phi \in \GLn$.
It is called $\GLn$-contravariant if
\begin{equation}
	\mu (\phi P) = |\det \phi|^{-1} \, \phi^{-t} \, \mu (P) \, \phi^{-1}
\end{equation}
for all $P \in \cP^n_o$ and each $\phi \in \GLn$.
The operator $M_2 \colon \cP^n_o \to \R^{n \times n}$ defined by
\[
	M_2(P) = \int_P x \, x^t \, dx
\]
for $P \in \cP^n_o$ calculates the moment matrix of the polytope $P$.
Clearly, it is a measurable valuation that is $\GLn$-covariant.
Now, we are in a position to formulate the already mentioned theorem,
which is a special case of a result by Ludwig \cite{Lud03}.
\begin{theorem}\label{th: P^n_o GL(n)-covariant matrix-valued symmetric}
	Assume that $\mu \colon \cP^n_o \to \R^{n \times n}$ is a measurable valuation which is $\GLn$-covariant.
	Furthermore, assume that $\mu (P)$ is a symmetric matrix for all $P \in \cP^n_o$.
	Then there exists a $k \in \R$ such that
	\[
		\mu (P) = k \, M_2(P)
	\]
	for all $P \in \cP^n_o$.
\end{theorem}
As was mentioned before, we will need an even stronger version of this theorem.
In fact, we have to remove the symmetry assumption.
That this is possible will be shown in the following sections.

Finally, we will now prove Theorem \ref{th: main2} from the introduction.
Note that one part of the \lq if and only if\rq-statement in Theorem \ref{th: main2} is trivial.
Here, and also in the rest of the article, we will therefore only prove the nontrivial parts of the statements from the introduction.

\begin{theorem}\label{th: P^n_o SL(n)-invariant}
	Let $n \geq 2$.
	Assume that $\mu \colon \cP^n_o \to \R$ is a measurable valuation which is $\SLn$-invariant.
	Then there exist constants $k_0, k_1, k_2 \in \R$ such that
	\[
		\mu (P) = k_0 \chi (P) + k_1 V(P) + k_2 V(P^*)
	\]
	for all $P \in \cP^n_o$.
\end{theorem}
\begin{proof}
	The proof of Theorem \ref{th: main0} from \cite{habpar13}
	uses the upper semicontinuity only at one point in \cite[Lemma 3.6]{habpar13}.
	To weaken the assumption of upper semicontinuity to measurability
	at this particular point, it suffices to show the following.

	Let $G \colon \R^2 \to \R$ be a measurable and antisymmetric function such that
	\begin{equation}\label{eq: G 0}
		G(s, 0) = G(0, s) = G(s, -s) = 0
	\end{equation}
	for all $s \in \R$.
	Moreover, suppose that for $s, t, u, v \in \R$ the quantity
	\begin{equation}\label{eq: G r independent}
	  G(t + r, v - r) - G(s + r, v - r) + G(s + r, u - r) - G(t + r, u - r)
	\end{equation}
	is independent of $r \in\R$.
	We want to show that each $G$ with these properties must vanish on $\R^2$.
	
	Associated with such a function $G$, define another function $H \colon \cP^1 \times \cP^1 \to \R$ by
	\[
		H([s, t], [u, v]) = G(t, v) - G(s, v) + G(s, u) - G(t, u)
	\]
	for all $s, t, u, v \in \R$.
	One can think of $H$ as a function on rectangles parallel to the coordinate axes.
	$G$ can be viewed as a function on the corners of those rectangles.
	Note that on the right hand side of the last equation the corners of this rectangle
	are traversed counter-clockwise.
	From \eqref{eq: G r independent} we obtain
	\begin{equation}\label{eq: P^n_o SL(n)-invariant - H diagonal}
		H([s, t] + r, [u, v] - r) = H([s, t], [u, v])
	\end{equation}
	for all $r, s, t, u, v \in \R$.
	Plugging $s = 0$, $u = 0$ and $r = v$ into \eqref{eq: P^n_o SL(n)-invariant - H diagonal}, we obtain
	\[
		H([v, t + v], [-v, 0]) = H([0, t], [0, v]) ,
	\]
	which, by \eqref{eq: G 0}, simplifies to
	\begin{equation}\label{eq: P^n_o SL(n)-invariant - -G(t + v, -v) = G(t, v)}
		-G(t + v, -v) = G(t, v) .
	\end{equation}
	Plugging $t = -v$, $u = 0$ and $r = v - s$ into \eqref{eq: P^n_o SL(n)-invariant - H diagonal} and using
	\eqref{eq: G 0} again, we obtain
	\[
		H([v, -s], [s - v, s]) = H([s, -v], [0, v]) .
	\]
	In terms of $G$ the last equality reads as
	\[
		-G(v, s) + G(v, s - v) - G(-s, s - v) = -G(s, v) .
	\]
	By the antisymmetry of $G$ and \eqref{eq: P^n_o SL(n)-invariant - -G(t + v, -v) = G(t, v)}
	this is equivalent to
	\[
		G(s, -v) = -G(s, v) .
	\]
	A glance at the definition of $H$ reveals that therefore
	\[
		H([s, t], -[u, v]) = -H([s, t], [u, v]) .
	\]
	Combining this with \eqref{eq: P^n_o SL(n)-invariant - H diagonal} we obtain
	\[
		H([s, t] + r, [u, v] + r) = H([s, t], [u, v]) ,
	\]
	which, in combination with \eqref{eq: P^n_o SL(n)-invariant - H diagonal}, yields
	\[
		H([s, t] + 2 r, [u, v]) = H([s, t], [u, v]) .
	\]
	In particular, this equality implies that
	\begin{equation}\label{eq: P^n_o SL(n)-invariant - H horizontal}
		H([s, s + t], [u, v]) = H([0, t], [u, v]) .
	\end{equation}
	Using the definition of $H$ in terms of $G$ it is easy to check that
	\[
		H([0, s + t], [u, v]) = H([0, s], [u, v]) + H([s, s + t], [u, v]) .
	\]
	Now \eqref{eq: P^n_o SL(n)-invariant - H horizontal} implies that
	$t \mapsto G(t, v) = H([0, t], [0, v])$ satisfies Cauchy's functional equation for every $v > 0$.
	Since $G$ is measurable, this function is linear.
	But since $G(v, v) = 0$, $G$ vanishes on $\R^2$.
\end{proof}

%% file: proof.tex
\subsection{The $1$-dimensional case}

In dimension one, a map $\mu \colon \cP^1_o \to \R$ is even if and only if it
is $\VLn[1]$-signum-covariant.
Similarly, $\mu$ is odd if and only if it is $\VLn[1]$-covariant.
Consequently, representation \eqref{eq: mu = mu^+ + mu^-} corresponds to
the standard decomposition of $\mu$ into its even and odd part.
The following two theorems classify all even and odd valuations $\mu \colon \cP^1_o \to \R$, respectively.
By decomposition \eqref{eq: mu = mu^+ + mu^-} these results give a complete picture of such valuations.
Recall that for $a,b,c,d > 0$, we denote by $I$ and $J$ the line segments
$I = [-a e_1, b e_1]$ and $J = [-c e_n, d e_n]$.

\begin{theorem}\label{th: dim 1 even}
	Assume that $\mu \colon \cP^1_o \to \R$ is an even valuation.
	Then, for all $a,b > 0$,
	\[
		\mu (I) = F(a) + F(b) ,
	\]
	where $F(r) := \frac 1 2 \mu [-r, r]$, $r \in (0, \infty)$.
	Moreover, if $\mu$ is $q$-homogeneous for some $q \in \R$, then
	\[
		\mu (I) = \mathrm{const} \cdot \left( a^q + b^q \right)
	\]
	for all $a,b > 0$.
\end{theorem}
\begin{proof}
	We refer to \cite{habpar13} for a proof of the first part of the statement.
	If $\mu$ is $q$-homogeneous, then
	\[
	  F(r) = \frac 1 2 \, \mu [-r, r] = r^q \, \frac 1 2 \, \mu [-1, 1] = r^q F(1).
	\]
	This immediately proves the second claim.
\end{proof}

\begin{theorem}\label{th: dim 1 odd}
	Assume that $\mu \colon \cP^1_o \to \R$ is an odd valuation.
	Then, for all $a,b > 0$,
	\[
		\mu (I) = F(b) - F(a) ,
	\]
	where $F(r) := \mu [-1, r]$, $r \in (0, \infty)$.
	Moreover, if $\mu$ is $q$-homogeneous for some $q \in \R \setminus \{ 0 \}$, then
	\[
		\mu (I) = \mathrm{const} \cdot \left( b^q - a^q \right)
	\]
	for all $a,b > 0$.
	If $\mu$ is $0$-homogeneous and measurable, then
	\[
		\mu (I) = \mathrm{const} \cdot \ln \left( \frac b a \right)
	\]
	for all $a,b > 0$.
\end{theorem}
\begin{proof}
	The proof of the first assertion can be found in \cite{habpar13}.
	In order to establish the part on the homogeneity we follow \cite{Lud02advval}.
	Let $\mu$ be $q$-homogeneous, $q \in \R$.
	By the first part of the theorem we obtain, for all $s,t > 0$,
	\[
		F(st) - F(s) = \mu [-s, st] = s^q \mu [-1, t] = s^q F(t) ,
	\]
	or, equivalently,
	\begin{equation} \label{eq: F(st) one dimension}
		F(st) = F(s) + s^q F(t) .
	\end{equation}
	For $q = 0$, the map $r \mapsto F(\exp(r))$ satisfies Cauchy's functional equation.
	If $\mu$ is measurable, we conclude $F(r) = \mathrm{const} \cdot \ln(r)$.
	For $q \neq 0$, we switch $s$ and $t$ in \eqref{eq: F(st) one dimension} and obtain
	\[
		F(s) + s^q F(t) = F(t) + t^q F(s) .
	\]
	Setting $t = 2$ and rearranging terms yields
	\[
		F(s) = \left( 1 - 2^q \right)^{-1} F(2) \left( 1 - s^q \right) ,
	\]
	which completes the proof.
\end{proof}

\subsection{The $2$-dimensional case}

Let $\mu \colon \cQ^2 \to \R^2$ be a valuation.
We say that $\mu$ splits over pyramids if there is a map $\muTr$ such that
\[
	\mu [I, J] = \muTr [I, -c e_2] + \muTr [I, d e_2] .
\]
In other words, the value of $\mu$ on the straight double pyramid $[I, J]$
is the sum of the values of $\muTr$ on the lower straight pyramid $[I, -c e_2]$ and the upper straight pyramid $[I, d e_2]$.
Our first result classifies $\VLn[2]$-covariant valuations on $\cQ^2$ and reveals that such valuations split over pyramids.

\begin{lemma}\label{le: Q^2 VL(2)-covariant}
	Assume that $\mu \colon \cQ^2 \to \R^2$ is an $\VLn[2]$-covariant valuation.
	Then
	\[
		\renewcommand{\arraystretch}{1.5}
		\mu [I, J] =
		\begin{pmatrix}
				- \frac 1 c & \frac 1 c & - \frac 1 d & \frac 1 d \\
				- \frac 1 a & - \frac 1 b & \frac 1 a & \frac 1 b \\
		\end{pmatrix}
		\cdot
		\begin{pmatrix}
				F(ac) \\
				F(bc) \\
				F(ad) \\
				F(bd) \\
		\end{pmatrix}
	\]
	for all $a,b,c,d > 0$, where $F(r) := \frac 1 2 \mu_1 [-e_1, r e_1, E_2]$, $r \in (0, \infty)$.
	In particular, $\mu$ splits over pyramids with
	\begin{equation}\label{spec. splitting over pyramid}
		\renewcommand{\arraystretch}{1.5}
		\muTr [I, -c e_2] :=
		\begin{pmatrix}
				\frac 1 c ( F(bc) - F(ac) ) \\
				- \frac 1 a F(ac) - \frac 1 b F(bc) \\
		\end{pmatrix}
		\quad \textrm{ and } \quad
		\muTr [I, d e_2] :=
		\begin{pmatrix}
				\frac 1 d ( F(bd) - F(ad) ) \\
				\frac 1 a F(ad) + \frac 1 b F(bd) \\
		\end{pmatrix} .
	\end{equation}
\end{lemma}
\begin{proof}
	Clearly, the first component $\mu_1 [I, J]$ is a valuation in each of its arguments.
	By the $\VLn[2]$-covariance it is odd in $I$ and even in $J$.
	Using Theorem \ref{th: dim 1 even}, the $\SLn[2]$-covariance and Theorem \ref{th: dim 1 odd},
	as well as the definition of $F$, we obtain
	\begin{align*}
		\mu_1 [I, J]
		&= \frac 1 2 \mu_1 [I, c E_2] + \frac 1 2 \mu_1 [I, d E_2] \\
		&= \frac 1 2 \mu_1 \left( \ptwotwo {\frac 1 c} {0} {0} {c} [c I, E_2] \right) +
		\frac 1 2 \mu_1 \left( \ptwotwo {\frac 1 d} {0} {0} {d} [d I, E_2] \right) \\
		&= \frac 1 {2c} \mu_1 [c I, E_2] + \frac 1 {2d} \mu_1 [d I, E_2] \\
		&= \frac 1 c \left( F(bc) - F(ac) \right) + \frac 1 d \left( F(bd) - F(ad) \right) .
	\end{align*}
	Similarly, the second component $\mu_2 [I, J]$ is a valuation in each of its arguments,
	but it is even in $I$ and odd in $J$.
	By the same arguments as before, we arrive at
	\[
		\mu_2 [I,J] = \frac 1 a \left( G(ad) - G(ac) \right) + \frac 1 b \left( G(bd) - G(bc) \right) ,
	\]
	where $G(r) := \frac 1 2 \mu_2 [E_1, -e_2, r e_2]$.
	Finally, by the $\VLn[2]$-covariance, $G = F$.
\end{proof}

We will now consider maps $\mu \colon \cR^2 \to \R^2$.
In order to obtain a complete classification of $\VLn[2]$-covariant measurable valuations on $\cR^2$,
we first need to establish some preliminary results.

\begin{lemma}\label{le: f^I independent}
	Assume that $\mu \colon \cR^2 \to \R^2$ is an $\SLn[2]$-covariant valuation which splits over pyramids.
	Then, for $a,b > 0$ and $x,y \in \R$, the function
	\[
		f^I(x, y) := \mu \left[ I, -c \ptwo x 1, d \ptwo y 1 \right]
		- \ptwotwo 1 0 x 1 \muTr [I, -c e_2]
		- \ptwotwo 1 0 y 1 \muTr [I, d e_2] ,
	\]
	is independent of $c,d > 0$ as long as $a, b, c, d, x, y$ form a double pyramid.
	Moreover,
	\begin{equation} \label{eq: f^I(x, y) one}
		f^I(x, y) = f^I(x, 0) + f^I(0, y)
	\end{equation}
	and
	\begin{equation} \label{eq: f^I(x, y) two}
		f^I(x, y) = \ptwotwo 1 0 y 1 f^I(x-y, 0) .
	\end{equation}
\end{lemma}
\begin{proof}
	For given $a, b, x, y$ choose $c$ and $d$ so small that
	$a, b, c, d, x, y$ form a double pyramid.
	By the valuation property we have
	\[
		\mu \left[ I, -c \ptwo x 1, d \ptwo y 1 \right] + \mu \left[ I, -s \ptwo y 1, t \ptwo y 1 \right] =
		\mu \left[ I, -c \ptwo x 1, t \ptwo y 1 \right] + \mu \left[ I, -s \ptwo y 1, d \ptwo y 1 \right]
	\]
	for sufficiently small $s,t > 0$.
	The $\SLn[2]$-covariance implies
	\begin{multline*}
		\mu \left[ I, -c \ptwo x 1, d \ptwo y 1 \right] + \ptwotwo 1 0 y 1 \mu [ I, -s e_2, t e_2] = \\
		\mu \left[ I, -c \ptwo x 1, t \ptwo y 1 \right] + \ptwotwo 1 0 y 1 \mu [ I, -s e_2, d e_2] .
	\end{multline*}
	Since $\mu$ splits over pyramids, this simplifies to
	\[
		\mu \left[ I, -c \ptwo x 1, d \ptwo y 1 \right] + \ptwotwo 1 0 y 1 \muTr [I, t e_2] =
		\mu \left[ I, -c \ptwo x 1, t \ptwo y 1 \right] + \ptwotwo 1 0 y 1 \muTr [I, d e_2].
	\]
	We conclude that the expression
	\[
		\mu \left[ I, -c \ptwo x 1, d \ptwo y 1 \right] - \ptwotwo 1 0 y 1 \muTr [I, d e_2]
	\]
	is independent of $d$.
	Similarly we see that the expression
	\[
		\mu \left[ I, -c \ptwo x 1, d \ptwo y 1 \right] - \ptwotwo 1 0 x 1 \muTr [I, -c e_2]
	\]
	is independent of $c$.
	This proves that $f^I$ is indeed well defined.
	For sufficiently small $r > 0$, the valuation property implies that
	\[
		\mu \left[ I, -c \ptwo x 1, d \ptwo y 1 \right] + \mu [ I, -r e_2, r e_2] =
		\mu \left[ I, -c \ptwo x 1, r e_2 \right] + \mu \left[ I, -r e_2, d \ptwo y 1 \right] .
	\]
	Since $f^I(0, 0) = 0$, this proves \eqref{eq: f^I(x, y) one}.
	The $\SLn[2]$-covariance and the definition of $f^I$ yield
	\begin{align*}
		\mu \left[ I, -c \ptwo x 1, d \ptwo y 1 \right]
		&= \ptwotwo 1 0 y 1 \mu \left[ I, -c \ptwo {x-y} 1, d e_2 \right] \\
		&= \ptwotwo 1 0 y 1
		\left( f^I(x-y, 0) + \ptwotwo 1 0 {x-y} 1 \muTr [I, -c e_2] + \muTr [I, d e_2] \right) \\
		&= \ptwotwo 1 0 y 1 f^I(x-y, 0) + \ptwotwo 1 0 x 1 \muTr [I, -c e_2] + \ptwotwo 1 0 y 1 \muTr [I, d e_2] ,
	\end{align*}
	which implies \eqref{eq: f^I(x, y) two}.
\end{proof}

\begin{lemma}\label{le: f^I VL(2)-covariant}
	Assume that $\mu \colon \cR^2 \to \R^2$ is a measurable valuation which is $\VLn[2]$-covariant.
	Let $f^I$ be defined as in Lemma \ref{le: f^I independent}.
	Then there exists a $\tilde k \in \R$ such that
	\[
		f^I(x, y) = \tilde k \left( \frac 1 a + \frac 1 b \right) (x-y) e_1
	\]
	for all $a, b > 0$ and $x, y \in \R$.
\end{lemma}
\begin{proof}
	We use $\muTr$ from Lemma \ref{le: Q^2 VL(2)-covariant} and apply Lemma \ref{le: f^I independent} to $\mu$.
	Combining \eqref{eq: f^I(x, y) one} and \eqref{eq: f^I(x, y) two}, we see that
	\[
		g(x) := f^I_2(x, 0) = \mu_2 \left[ I, -c \ptwo x 1, d e_2 \right] - \mu_2 [I, J]
	\]
	satisfies Cauchy's functional equation.
	Since $\mu$ is measurable, so is $g$.
	Therefore $g$ is linear.
	Thus there exists a $\nu \colon \cP^1_o \to \R$ with
	\[
		\mu_2 \left[ I, -c \ptwo x 1, d e_2 \right] - \mu_2 [I, J]  = \nu (I) x .
	\]
	Using the $\VLn[2]$-covariance for $\ptwotwo {-1} 0 0 1$, we obtain
	\[
		\nu (-I) x
		= \mu_2 \left[ -I, -c \ptwo x 1, d e_2 \right] - \mu_2 [-I, J]
		= \mu_2 \left[ I, -c \ptwo {-x} 1, d e_2 \right] - \mu_2 [I, J]
		= \nu (I) (-x) .
	\]
	Consequently, $\nu$ is odd.
	From the definition of $\nu$, the $\SLn[2]$-covariance of $\mu$ for $\ptwotwo {-1} 0 0 {-1}$,
	the definition of $f^I$, relation \eqref{eq: f^I(x, y) two} and again the definition of $\nu$ we infer
	\begin{align*}
		\nu (-I) x
		&= \mu_2 \left[ -I, -c \ptwo x 1, d e_2 \right] - \mu_2 [-I, J] \\
		&= - \mu_2 \left[ I, -d e_2, c \ptwo x 1 \right] + \mu_2 [I, -J] \\
		&= - f^I_2(0,x) \\
		&= - f^I_2(-x,0) \\
		&= - \nu (I) (-x) ,
	\end{align*}
	i.e.\ $\nu$ is even.
	Since $\nu$ is odd and even, it has to vanish.
	From \eqref{eq: f^I(x, y) two} we deduce that $f^I_2(x, y) = 0$.
	
	Using \eqref{eq: f^I(x, y) one}, \eqref{eq: f^I(x, y) two} and what we have just shown,
	we see that
	\[
		h(x) := f^I_1(x, 0)
		= \mu_1 \left[ I, -c \ptwo x 1, d e_2 \right] - \mu_1 [I, J] - x \muTr_2 [I, -c e_2]
	\]
	satisfies Cauchy's functional equation.
	Since $\mu$ is measurable, so is $h$.
	Therefore $h$ is linear.
	Thus there exists a $\xi \colon \cP^1_o \to \R$ with
	\[
		\mu_1 \left[ I, -c \ptwo x 1, d e_2 \right] - \mu_1 [I, J] - x \muTr_2 [I, -c e_2] = \xi (I) x .
	\]
	Using the definition of $\xi$, the $\VLn[2]$-covariance for $\ptwotwo {-1} 0 0 1$ and the representation of
	$\muTr_2 [I, -r c e_2]$ from Lemma \ref{le: Q^2 VL(2)-covariant}, we obtain
	\begin{align*}
		\xi (-I) x
		&= \mu_1 \left[ -I, -c \ptwo x 1, d e_2 \right] - \mu_1 [-I, J] - x \muTr_2 [-I, -c e_2] \\
		&= -\mu_1 \left[ I, -c \ptwo {-x} 1, d e_2 \right] + \mu_1 [I, J] - x \muTr_2 [I, -c e_2] \\
		&= -\xi (I) (-x) ,
	\end{align*}
	i.e.\ $\xi$ is even.
	Using the $\SLn[2]$-covariance of $\mu$ for $\ptwotwo r 0 0 {\frac 1 r}$ and again Lemma \ref{le: Q^2 VL(2)-covariant}, yields
	\begin{align*}
		\xi (r I) x
		&= \mu_1 \left[ r I, -c \ptwo x 1, d e_2 \right] - \mu_1 [r I, J] - x \muTr_2 [r I, -c e_2] \\
		&= r \mu_1 \left[ I, -c \ptwo {\frac x r} r, r d e_2 \right]
		- r \mu_1 [I, r J] - \frac x r \muTr_2 [I, -r c e_2] \\
		&= r \mu_1 \left[ I, -r c \ptwo {\frac x {r^2}} 1, r d e_2 \right]
		- r \mu_1 [I, r J] - r \frac x {r^2} \muTr_2 [I, -r c e_2] \\
		&= r \xi (I) \frac x {r^2} ,
	\end{align*}
	i.e.\ $\xi$ is $(-1)$-homogeneous.
	Clearly, $\xi$ is a valuation.
	By Theorem \ref{th: dim 1 even} there exists a $\tilde k \in \R$ such that $\xi(I) = \tilde k (\frac 1 a + \frac 1 b)$.
	An application of \eqref{eq: f^I(x, y) two} completes the proof.
\end{proof}

A combination of Lemma \ref{le: Q^2 VL(2)-covariant}, Lemma \ref{le: f^I independent},
Lemma \ref{le: f^I VL(2)-covariant} and relation \eqref{spec. splitting over pyramid} proves the following.

\begin{corollary}\label{cor. representation mu}
	Assume that $\mu \colon \cR^2 \to \R^2$ is a measurable valuation which is $\VLn[2]$-covariant.
	Then there exists a $\tilde k \in \R$ such that
	\begin{align*}
		\mu \left[ I, -c \ptwo x 1, d \ptwo y 1 \right]
		= & \ptwotwo 1 0 x 1
		\renewcommand{\arraystretch}{1.5}
		\begin{pmatrix}
			\frac 1 c ( F(bc) - F(ac) ) \\
			- \frac 1 a F(ac) - \frac 1 b F(bc) \\
		\end{pmatrix}
		\renewcommand{\arraystretch}{1.0}
		+ \ptwotwo 1 0 y 1
		\renewcommand{\arraystretch}{1.5}
		\begin{pmatrix}
			\frac 1 d ( F(bd) - F(ad) ) \\
			\frac 1 a F(ad) + \frac 1 b F(bd) \\
		\end{pmatrix} \\
		& + \tilde k \left( \frac 1 a + \frac 1 b \right) (x-y) e_1
	\end{align*}
	for all $a, b, c, d > 0$ and $x, y \in \R$, whenever $a,b,c,d,x,y$ form a double pyramid,
	where $F(r) := \frac 1 2 \mu_1 [-e_1, r e_1, E_2]$, $r \in (0, \infty)$.
\end{corollary}

Now, we have all prerequisites to classify valuations on $\cR^2$ which are $\VLn[2]$-covariant and measurable.

\begin{lemma}\label{le: R^2 VL(2)-covariant}
	Assume that $\mu \colon \cR^2 \to \R^2$ is a measurable valuation which is $\VLn[2]$-covariant.
	Then there exists a $k \in \R$ such that
	\begin{equation}\label{eq: R^2 VL(2)-covariant}
		\mu \left[ I, -c \ptwo x 1, d \ptwo y 1 \right]
		= k (a + b) \left( (c + d) \ptwo {b - a} {d - c} + (y d^2 - x c^2) e_1 \right)
	\end{equation}
	for all $a, b, c, d > 0$ and $x, y \in \R$, whenever $a,b,c,d,x,y$ form a double pyramid.
\end{lemma}
\begin{proof}
	Let $s, t > 0$ and define a triangle $S$ by
	\[
		S = \left[ -s e_1, - e_2, \ptwo s t \right] .
	\]
	Based on different representations of $S$, we will calculate the first component $\mu_1 (S)$ of $\mu (S)$ in two different ways.
	First, note that
	\[
		S = \left[ -s e_1, \frac s {1 + t} e_1, - e_2, t \ptwo {\frac s t} 1 \right].
	\]
	From Corollary \ref{cor. representation mu} we deduce
	\begin{align*}
		\mu_1 (S)
		&= F \left( \frac s {1 + t} \right) - F(s) + \frac 1 t F \left( \frac {st} {1 + t} \right)  - \frac 1 t F(st) \\
		&\phantom{=} {} + \frac s t \left( \frac 1 s F(st) + \frac {1 + t} s F \left( \frac {st} {1 + t} \right) \right)
		- \tilde k \left( \frac 1 s + \frac {1 + t} s \right) \frac s t \\
		&= F \left( \frac s {1 + t} \right) - F(s) + \frac {2 + t} t F \left( \frac {st} {1 + t} \right) - \frac {2 + t} t \tilde k
	\end{align*}
	for some unknown function $F \colon (0, \infty) \to \R$ and some unknown constant $\tilde k \in \R$.
	Second, we have the representation
	\[
	  S = \ptwotwo 0 1 {-1} 0 \left[ - e_1, \frac t 2 e_1, -s \ptwo {- \frac t s} 1, s e_2 \right] .
	\]
	From the $\SLn[2]$-covariance of $\mu$ and Corollary \ref{cor. representation mu} we obtain
	\[
		\mu_1 (S)
		= - \mu_2 \left[ - e_1, \frac t 2 e_1, -s \ptwo {- \frac t s} 1, s e_2 \right]
		= F(s) + \frac 2 t F \left( \frac {st} 2 \right)
		- F(s) - \frac 2 t F \left( \frac {st} 2 \right)
		= 0 .
	\]
	Combining the above representations of $\mu_1(S)$, yields
	\begin{equation}\label{eq: functional equation}
		F(s) = F \left( \frac s {1 + t} \right) + \frac {2 + t} t F \left( \frac {st} {1 + t} \right) - \frac {2 + t} t \tilde k .
	\end{equation}
	For fixed $\tilde k$ this is an inhomogeneous functional equation in $F$.
	Clearly, $F(r) = \tilde k$, $r \in (0, \infty)$, is a solution.
	So it remains to solve the homogeneous functional equation
	\[
		G(s) = G \left( \frac s {1 + t} \right) + \frac {2 + t} t G \left( \frac {st} {1 + t} \right)
	\]
	for an unknown function $G \colon (0, \infty) \to \R$.
	Setting $s = u + v$ and $t = \frac u v$, $u, v > 0$, we obtain
	\[
		G(u + v) = G(v) + \frac {2v + u} u G(u) .
	\]
	On the other hand, setting $s = u + v$ and $t = \frac v u$, we arrive at
	\[
		G(u + v) = G(u) + \frac {2u + v} v G(v) .
	\]
	Combining the last two equations we obtain
	\[
		\frac v u G(u) = \frac u v G(v) .
	\]
	Setting $v = 1$ finally gives
	\[
		G(u) = u^2 G(1) .
	\]
	We see that $F(r) = k r^2 + \tilde k$ for some $k \in \R$ is the general solution for \eqref{eq: functional equation}.
	For this particular $F$, Corollary \ref{cor. representation mu} immediately proves \eqref{eq: R^2 VL(2)-covariant}.
	Note that when calculating \eqref{eq: R^2 VL(2)-covariant} in this way, all terms containing $\tilde k$ cancel out.
\end{proof}

Finally, we are in a position to prove our main characterization theorems in the plane.
Let us start with the $\VLn[2]$-covariant case.

\begin{theorem}\label{th: P^2_o VL(2)-covariant}
	Assume that $\mu \colon \cP^2_o \to \R^2$ is a measurable valuation which is $\VLn[2]$-covariant.
	Then there exists a $k \in \R$ such that
	\[
		\mu (P) = k \, m (P)
	\]
	for all $P \in \cP^2_o$.
\end{theorem}
\begin{proof}
	From Lemma \ref{le: R^2 VL(2)-covariant} we deduce that the vector space of measurable $\VLn[2]$-covariant
	valuations $\mu \colon \cR^2 \to \R^2$ is at most $1$-dimensional.
	Since the moment vector $m$ is a measurable $\VLn[2]$-covariant valuation on $\cR^2$,
	there exists a constant $k \in \R$ with $\mu ( P ) = k \, m( P )$ for all double pyramids $P \in \cR^2$.
	Since $\mu$ and $m$ are both $\SLn[2]$-covariant, the last equality actually holds for all
	$\SLn[2]$ images of elements in $\cR^2$.
	Theorem \ref{th: determined by values on SL(n)(R^n)} therefore concludes the proof.
\end{proof}

Next, the $\VLn[2]$-signum-covariant case will be settled.

\begin{theorem}\label{th: P^2_o VL(2)-signum-covariant}
	Assume that $\mu \colon \cP^2_o \to \R^2$ is a measurable valuation which is $\VLn[2]$-signum-covariant.
	Then there exists a $k \in \R$ such that
	\[
		\mu (P) = k \, \rho_{\frac \pi 2} \, m (P^*)
	\]
	for all $P \in \cP^2_o$.
\end{theorem}
\begin{proof}
	Define $\nu \colon \cP^2_o \to \R^2$ by
	\[
		\nu (P) = \rho_{\frac \pi 2}^{-1} \mu (P^*)
	\]
	for all $P \in \cP^2_o$.
	Relations \eqref{valuation polar bodies} and \eqref{contravariance polarity} show that $\nu$ is
	a measurable $\VLn[2]$-covariant valuation.
	If Theorem \ref{th: P^2_o VL(2)-covariant} is applied to $\nu$, then the assertion follows easily
	from the fact that polarity is an involution.
\end{proof}

Combining the last two results, we can now prove the non-trivial part of Theorem \ref{th: main1 2d}.

\begin{theorem}
	Assume that $\mu \colon \cP^2_o \to \R^2$ is a measurable valuation which is $\SLn[2]$-covariant.
	Then there exist $k_1, k_2 \in \R$ such that
	\[
		\mu (P) = k_1 \, m(P) + k_2 \, \rho_{\frac \pi 2} \, m(P^*)
	\]
	for all $P \in \cP^2_o$.
\end{theorem}
\begin{proof}
	Define $\mu^+$ and $\mu^-$ as in \eqref{eq: definition mu^+} and \eqref{eq: definition mu^-}, respectively.
	Theorems \ref{th: P^2_o VL(2)-covariant} and \ref{th: P^2_o VL(2)-signum-covariant} and \eqref{eq: mu = mu^+ + mu^-}
	now directly imply the desired result.
\end{proof}

By the correspondance between co- and contravariant valuations via polarity from Section \ref{prelim},
the following theorem is equivalent to the previous one.

\begin{theorem}
	Assume that $\mu \colon \cP^2_o \to \R^2$ is a measurable valuation which is $\SLn[2]$-contravariant.
	Then there exist $k_1, k_2 \in \R$ such that
	\[
		\mu (P) = k_1 \, m(P^*) + k_2 \, \rho_{\frac \pi 2} \, m(P)
	\]
	for all $P \in \cP^2_o$.
\end{theorem}

At the end of this subsection we prove a generalization of Theorem \ref{th: P^n_o GL(n)-covariant matrix-valued symmetric}
in the $2$-dimensional case.
We will show that the symmetry assumption can be omitted.

\begin{theorem}\label{th: P^2_o GL(2)-covariant matrix-valued}
	Assume that $\mu \colon \cP^2_o \to \R^{2 \times 2}$ is a measurable valuation which is $\GLn[2]$-covariant.
	Then there exists a $k \in \R$ such that
	\[
		\mu (P) = k \, M_2(P)
	\]
	for all $P \in \cP^2_o$.
\end{theorem}
\begin{proof}
	We can write the map $P \mapsto \mu (P)$ as the sum of its symmetric part,
	$P \mapsto \frac 1 2 \left( \mu(P) + \mu(P)^t \right)$,
	and its antisymmetric part, $P \mapsto \frac 1 2 \left( \mu(P) - \mu(P)^t \right)$.
	Note that $P \mapsto \mu (P)^t$ inherits all the assumed properties of $\mu$.
	Therefore, the symmetric and antisymmetric part of $\mu$ are both measurable and
	$\GLn[2]$-covariant valuations.
	Hence, by Theorem \ref{th: P^n_o GL(n)-covariant matrix-valued symmetric},
	we only have to show that the antisymmetric part vanishes.

	Therefore, assume that $\mu(P)$ is antisymmetric for all $P \in \cP^2_o$.
	The component $\mu_{12} [I, J]$ is a $2$-homogeneous odd valuation in both $I$ and $J$.
	On the one hand, by Theorem \ref{th: dim 1 odd}, we have
	\[
		\mu_{12} [I, J] = \mathrm{const} \cdot \left( b^2 - a^2 \right) \left( d^2 - c^2 \right) .
	\]
	In particular, $\mu_{12} [I, J] = \mu_{12} [-c e_1, d e_1, -a e_2, b e_2]$.
	On the other hand, by the $\GLn[2]$-covariance and the antisymmetry of $\mu$,
	\begin{align*}
		\mu_{12} [I, J]
		&= \mu_{12} \left( \ptwotwo 0 1 1 0 [-c e_1, d e_1, -a e_2, b e_2] \right) \\
		&= \mu_{21} [-c e_1, d e_1, -a e_2, b e_2] \\
		&= - \mu_{12} [-c e_1, d e_1, -a e_2, b e_2] .
	\end{align*}
	Thus $\mu_{12}$ vanishes on $\cQ^2$.
	By the antisymmetry of $\mu$ we conclude that $\mu$ vanishes on $\cQ^2$.
	
	As in the proof of Lemma \ref{le: f^I independent} one can show that for fixed
	$a,b > 0$ and $x,y \in \R$,	the quantity
	\[
		\mu \left[ I, -c \ptwo x 1, d \ptwo y 1 \right]
	\]
	does not depend on $c,d >0$, as long as $a, b, c, d, x, y$ form a double pyramid.
	Hence for each $I$ the function
	\[
		f^I(x, y) := \mu \left[ I, -c \ptwo x 1, d \ptwo y 1 \right]
	\]
	is well defined on $\R \times \R$.
	As in the proof of Lemma \ref{le: f^I independent} we see that
	\[
		f^I(x, y) = f^I(x, 0) + f^I(0, y)
	\]
	and
	\[
		f^I(x, y) = \ptwotwo 1 0 y 1 f^I(x-y, 0) {\ptwotwo 1 0 y 1}^t .
	\]
	Combining these two equations and using the antisymmetry of $\mu$,
	we see that $x \mapsto f^I_{12}(x, 0)$ satisfies Cauchy's functional equation.
	Since $\mu$ is measurable, so is $f^I$, and hence $x \mapsto f^I_{12}(x, 0)$ is linear.
	The definition of $f^I$ therefore implies
	\begin{equation}\label{12 linearity}
	  \mu_{12} \left[ I, -c \ptwo {rx} 1, d e_2 \right] = r \mu_{12} \left[ I, -c \ptwo x 1, d e_2 \right]
	\end{equation}
	for $r > 0$.
	On the one hand, by the $\GLn[2]$-covariance and \eqref{12 linearity}, we obtain
	\begin{align*}
		\mu_{12} \left[ r I, -c \ptwo x 1, d e_2 \right]
		&= \mu_{12} \left( \ptwotwo r 0 0 1 \left[ I, -c \ptwo {\frac x r} 1, d e_2 \right] \right) \\
		&= r^2 \mu_{12} \left[ I, -c \ptwo {\frac x r} 1, d e_2 \right] \\
		&= r \mu_{12} \left[ I, -c \ptwo x 1, d e_2 \right] .
	\end{align*}
	On the other hand, the $\GLn[2]$-covariance and the independence with respect to $c$ and $d$ yield
	\[
		\mu_{12} \left[ r I, -c \ptwo x 1, d e_2 \right] =
		\mu_{12} \left( r \left[ I, - \frac c r \ptwo x 1, \frac d r e_2 \right] \right) =
		r^4 \mu_{12} \left[ I, -c \ptwo x 1, d e_2 \right] .
	\]
	We conclude that
	\[
		\mu_{12} \left[ I, -c \ptwo x 1, d e_2 \right] = 0.
	\]
	By the $\GLn[2]$-covariance we have
	\[
		\mu_{12} \left[ I, -c \ptwo x 1, d \ptwo y 1 \right] = \mu_{12} \left[ I, -c \ptwo {x-y} 1, d e_2 \right] = 0.
	\]
	From the antisymmetry we infer that $\mu$ vanishes on $\cR^2$.
	An application of Theorem \ref{th: determined by values on SL(n)(R^n)} completes the proof.
\end{proof}

As before, the following theorem is equivalent to the previous one.

\begin{theorem}\label{th: P^2_o GL(2)-contravariant matrix-valued}
	Assume that $\mu \colon \cP^2_o \to \R^{2 \times 2}$ is a measurable valuation which is $\GLn[2]$-contravariant.
	Then there exists a $k \in \R$ such that
	\[
		\mu (P) = k \, M_2(P^*)
	\]
	for all $P \in \cP^2_o$.
\end{theorem}

\subsection{The $n$-dimensional case}

In this subsection we will prove our main theorems in all dimensions greater or equal than three.
Let us start by formulating four theorems which will be established in the sequel.
First, we state covariant classification results for vector and matrix valued valuations.

\begin{theorem}\label{th: P^n_o SL(n)-covariant}
	Let $n \geq 2$.
	Assume that $\mu \colon \cP^n_o \to \R^n$ is a measurable valuation which is $\SLn$-covariant.
	Then, for $n = 2$, there exist $k_1, k_2 \in \R$ such that
	\[
		\mu (P) = k_1 \, m(P) + k_2 \, \rho_{\frac \pi 2} \, m(P^*)
	\]
	for all $P \in \cP^n_o$ and, for $n \geq 3$, there exists a $k \in \R$ such that
	\[
		\mu (P) = k \, m(P)
	\]
	for all $P \in \cP^n_o$.
\end{theorem}

\begin{theorem}\label{th: P^n_o GL(n)-covariant matrix-valued}
	Let $n \geq 2$.
	Assume that $\mu \colon \cP^n_o \to \R^{n \times n}$ is a measurable valuation which is $\GLn$-covariant.
	Then there exists a $k \in \R$ such that
	\[
		\mu (P) = k \, M_2(P)
	\]
	for all $P \in \cP^n_o$.
\end{theorem}

Second, we formulate the corresponding contravariant statements.

\begin{theorem}\label{th: P^n_o SL(n)-contravariant}
	Let $n \geq 2$.
	Assume that $\mu \colon \cP^n_o \to \R^n$ is a measurable valuation which is $\SLn$-contravariant.
	Then, for $n = 2$, there exist $k_1, k_2 \in \R$ such that
	\[
		\mu (P) = k_1 \, m(P^*) + k_2 \, \rho_{\frac \pi 2} \, m(P)
	\]
	for all $P \in \cP^n_o$ and, for $n \geq 3$, there exists a $k \in \R$ such that
	\[
		\mu (P) = k \, m(P^*)
	\]
	for all $P \in \cP^n_o$.
\end{theorem}

\begin{theorem}\label{th: P^n_o GL(n)-contravariant matrix-valued}
	Let $n \geq 2$.
	Assume that $\mu \colon \cP^n_o \to \R^{n \times n}$ is a measurable valuation which is $\GLn$-contravariant.
	Then there exists a $k \in \R$ such that
	\[
		\mu (P) = k \, M_2(P^*)
	\]
	for all $P \in \cP^n_o$.
\end{theorem}

Recall that the $2$-dimensional cases of the above results have already been established.
Obviously, our main Theorem \ref{th: main1} will be a direct consequence of Theorem \ref{th: P^n_o SL(n)-covariant}.

In full generality, the above theorems will be proved at the end of this section.
We will do this by induction on the dimension.
However, it is necessary to perform this induction simultaneously for all four theorems.
For the reader's convenience we therefore collect the main steps in several lemmas.
The induction itself can be found at the very end of this section.
Let us start with the vector valued case.

\begin{lemma}\label{le: induction R^n_o SL(n)-contravariant vanishes}
	Let $n \geq 3$ and suppose that Theorems \ref{th: P^n_o SL(n)-contravariant}
	and \ref{th: P^n_o GL(n)-contravariant matrix-valued} hold in dimension $n-1$.
	Assume that $\mu \colon \cR^n \to \R^n$ is a measurable valuation which is $\SLn$-contravariant and vanishes on $\cQ^n$.
	Then $\mu$ vanishes on $\cR^n$.
	More explicitly, if
	\[
		\mu [B, -c e_n, d e_n] = 0
	\]
	for all $B \in \cP^{n-1}_o$ and $c, d > 0$,
	then
	\[
		\mu \left[ B, -c \ptwo x 1, d \ptwo y 1 \right] = 0
	\]
	for all $B \in \cP^{n-1}_o$, $c, d > 0$ and $x, y \in \R^{n-1}$,
	whenever $B, c, d, x, y$ form a double pyramid.
\end{lemma}
\begin{proof}
	The valuation property implies that
	\begin{multline*}
		\mu \left[ B, -c \ptwo x 1, d \ptwo y 1 \right] + \mu \left[ B, -s \ptwo y 1, t \ptwo y 1 \right] = \\
		\mu \left[ B, -c \ptwo x 1, t \ptwo y 1 \right] + \mu \left[ B, -s \ptwo y 1, d \ptwo y 1 \right]
	\end{multline*}
	for sufficiently small $s,t > 0$.
	By the $\SLn$-contravariance and the assumption that $\mu$ vanishes on $\cQ^n$, we have
	\[
		\mu \left[ B, -s \ptwo y 1, t \ptwo y 1 \right] = 0
		\quad \textrm{ and } \quad
		\mu \left[ B, -s \ptwo y 1, d \ptwo y 1 \right] = 0 .
	\]
	Therefore
	\[
		\mu \left[ B, -c \ptwo x 1, d \ptwo y 1 \right] = \mu \left[ B, -c \ptwo x 1, t \ptwo y 1 \right] ,
	\]
	i.e.\ the left hand side is independent of $d$.
	Similarly we see that it is also independent of $c$.
	So for each $B \in \cP^{n-1}$ we can define a function $f^B \colon \R^{n-1} \times \R^{n-1} \to \R^n$ by
	\[
		f^B(x,y) = \mu \left[ B, -c \ptwo x 1, d \ptwo y 1 \right] ,
	\]
	as long as $B, c, d, x, y$ form a double pyramid.
	It remains to prove that $f^B$ vanishes for each $B \in \cP^{n-1}$.
	By the $\SLn$-contravariance we obtain
	\[
		\mu \left[ B, -c \ptwo x 1, d \ptwo y 1 \right] =
		{\ptwotwo {\Id'} 0 y 1}^{-t} \mu \left[ B, -c \ptwo {x-y} 1, d e_n \right] ,
	\]
	which is equivalent to
	\begin{equation} \label{eq: f^B(x-y, 0)}
		f^B(x, y) = \ptwotwo {\Id'} {-y^t} 0 1 f^B(x-y, 0) .
	\end{equation}
	The valuation property yields
	\[
		\mu \left[ B, -c \ptwo x 1, d \ptwo y 1 \right] + \mu [ B, -r e_n, r e_n] =
		\mu \left[ B, -c \ptwo x 1, r e_n \right] + \mu \left[ B, -r e_n, d \ptwo y 1 \right]
	\]
	for sufficiently small $r > 0$.
	Since $\mu [ B, -r e_n, r e_n] = 0$, relation \eqref{eq: f^B(x-y, 0)} implies
	\begin{equation} \label{eq: f^B(x, 0)}
		f^B(x+y, 0) = \ptwotwo {\Id'} {-y^t} 0 1 f^B(x, 0) + f^B(y, 0) .
	\end{equation}
	So the map $x \mapsto f^B_*(x,0)$ satisfies Cauchy's functional equation.
	The measurability of $\mu$ implies that also $f^B_*$ is measurable.
	Hence $x \mapsto f^B_*(x,0)$ is linear.
	Thus there exists a map $\nu \colon \cP^{n-1}_o \to \R^{(n-1) \times (n-1)}$ such that
	\begin{equation}\label{eq: f^B nu}
		f^B_*(x,0) = \nu (B) x .
	\end{equation}
	Using the $\SLn$-contravariance of $\mu$, we see that $\nu$ is $\GLn[n-1]$-contravariant.
	Note that in order to prove this for linear transformations with negative determinant,
	one has to use \eqref{eq: f^B(x-y, 0)}.
	By the assumption that Theorem \ref{th: P^n_o GL(n)-contravariant matrix-valued} holds in
	dimension $n-1$, there exists a $k \in \R$ such that
	\begin{equation}\label{eq: nu M_2}
		\nu (B) = k M_2'(B^*) .
	\end{equation}
	For $i \in \{ 1, \ldots, n-1 \}$, let $I_i$ be a line segment in $\spn \{ e_i \}$ containing the origin in its interior.
	Since $n \geq 3$, the double pyramid
	\[
	  \left[ I_1, \ldots, I_{n-1}, -c \ptwo {e_1'} 1, d e_n \right]
	\]
	is contained in the $\SLn$-image of $\cQ^n$.
	From relations \eqref{eq: f^B nu}, \eqref{eq: nu M_2} and the definition of $f^B_*$, we obtain
	\[
		k M_2' \left( [I_1, \ldots, I_{n-1}]^* \right) e_1' =
		\mu_* \left[ I_1, \ldots, I_{n-1}, -c \ptwo {e_1'} 1, d e_n \right] =
		0 .
	\]
	Therefore, $k = 0$ and, by \eqref{eq: f^B(x-y, 0)}, \eqref{eq: f^B nu} and \eqref{eq: nu M_2},
	we conclude that $f^B_*$ vanishes.

	Using what we have just shown,
	the same approach yields the existence of a map $\xi \colon \cP^{n-1}_o \to \R^{n-1}$ such that
	\[
		f^B_n(x,0) = {\xi (B)}^t x .
	\]
	The $\SLn$-contravariance of $\mu$ implies that $\xi$ is $\SLn[n-1]$-contravariant and $(-1)$-homogeneous.
	Note that the moment vector is homogeneous of degree $n+1$.
	By the assumption that Theorem \ref{th: P^n_o SL(n)-contravariant} holds in dimension $n-1$,
	we deduce that $\xi$ vanishes on $\cP^{n-1}_o$.
	Therefore, using \eqref{eq: f^B(x-y, 0)}, also $f^B_n$ vanishes.
\end{proof}

\begin{lemma}\label{le: induction P^n_o SL(n)-contravariant}
	Let $n \geq 3$ and suppose that Theorems \ref{th: P^n_o SL(n)-contravariant}
	and \ref{th: P^n_o GL(n)-contravariant matrix-valued} hold in dimension $n-1$.
	Assume that $\mu \colon \cP^n_o \to \R^n$ is a measurable valuation which is $\SLn$-contravariant.
	Then there exists a $k \in \R$ such that
	\[
		\mu (P) = k \, m(P^*)
	\]
	for all $P \in \cP^n_o$.
\end{lemma}
\begin{proof}
	The expression $\mu_n [B,J]$ is a measurable valuation in both arguments.
	Since it is $\SLn[n-1]$-invariant in $B$, Theorem \ref{th: P^n_o SL(n)-invariant} implies that
	\[
		\mu_n [B,J] = \nu_0 (J) + \nu_1 (J) V'(B) + \nu_2 (J) V'(B^*) ,
	\]
	with suitable $\nu_0, \nu_1, \nu_2 \colon \cP^1_o \to \R$.
	For an arbitrary $\theta \in \VLn[n-1] \setminus \SLn[n-1]$, we have
	\begin{multline*}
		\nu_0 (-J) + \nu_1 (-J) V'(B) + \nu_2 (-J) V'(B^*) \\
		\begin{aligned}
			&= \mu_n [B,-J] \\
			&= - \mu_n [\theta B, J] \\
			&= - \left( \nu_0 (J) + \nu_1 (J) V'(\theta B) + \nu_2 (J) V'((\theta B)^*) \right) \\
			&= - \nu_0 (J) - \nu_1 (J) V'(B) - \nu_2 (J) V'(B^*) ,
		\end{aligned}
	\end{multline*}
	where we used the invariance of volume with respect to maps with determinant $-1$ and \eqref{contravariance polarity}.
	Comparing degrees of homogeneity with repect to $B$ shows that $\nu_0$, $\nu_1$ and $\nu_2$ are odd.
	Similarly we can check that $\nu_0$, $\nu_1$ and $\nu_2$ are measurable valuations
	and that $\nu_0$ is $(-1)$-homogeneous, $\nu_1$ $0$-homogeneous and $\nu_2$ $(-2)$-homogeneous.
	By Theorem \ref{th: dim 1 odd} there exist constants $k_0, k_1, k_2 \in \R$ such that
	\begin{equation}\label{eq: induction P^n_o SL(n)-contravariant - mu_n}
		\mu_n [B,J] = k_0 \left( d^{-1} - c^{-1} \right)
			+ k_1 \ln \left( \frac d c \right) V'(B) + k_2 \left( d^{-2} - c^{-2} \right) V'(B^*) .
	\end{equation}
	The expression $\mu_* [B,J]$ is also a measurable valuation in both arguments.
	Since it is $\SLn[n-1]$-contravariant in $B$, the assumption that Theorem \ref{th: P^n_o SL(n)-contravariant}
	holds in dimension $n-1$ implies that
	\[
		\mu_* [B,J] = \xi (J) m'(B^*) + \delta_{n, 3} \tilde \xi (J) \rho_{\frac \pi 2} m'(B) ,
	\]
	where $\xi, \tilde \xi \colon \cP^1_o \to \R$ and $\delta_{n, 3}$ denotes the Kronecker delta.
	For an arbitrary map $\theta \in \VLn[n-1] \setminus \SLn[n-1]$ we have
	\begin{align*}
		\xi (-J) m'(B^*) + \delta_{n, 3} \tilde \xi (-J) \rho_{\frac \pi 2} m'(B)
		&= \mu_* [B,-J] \\
		&= \theta^t \mu_* [\theta B, J] \\
		&= \theta^t \left( \xi (J) m'((\theta B)^*) + \delta_{n, 3} \tilde \xi (J) \rho_{\frac \pi 2} m'(\theta B) \right) \\
		&= \xi (J) m'(B^*) - \delta_{n, 3} \tilde \xi (J) \rho_{\frac \pi 2} m'(B) ,
	\end{align*}
	where we used the covariance of moment vectors and \eqref{contravariance polarity}.
	By comparing degrees of homogeneity with respect to $B$, we see that $\xi$ is even and $\tilde \xi$ odd.
	Similarly, we can check that $\xi$ and $\tilde \xi$ are measurable valuations
	and that $\xi$ is $(-1)$-homogeneous and $\tilde \xi$ $2$-homogeneous,
	where we used that $\tilde \xi$ only shows up for $n = 3$.
	By Theorems \ref{th: dim 1 even} and \ref{th: dim 1 odd} there exist constants $k_3, \tilde k_3 \in \R$ such that
	\begin{equation}\label{eq: induction P^n_o SL(n)-contravariant - mu_*}
		\mu_* [B,J] = k_3 \left( c^{-1} + d^{-1} \right) m'(B^*)
			+ \delta_{n, 3} \tilde k_3 \left( d^2 - c^2 \right) \rho_{\frac \pi 2} m'(B) .
	\end{equation}
	For $i \in \{ 1, \ldots, n \}$, let $I_i := [-a_i e_i, b_i e_i]$, $a_i, b_i > 0$.
	The $\SLn$-contravariance of $\mu$ yields
	\[
		\mu_n [I_1, \ldots, I_{n-2}, I_{n-1}, r I_n] =
		\mu_{n-1} [I_1, \ldots, I_{n-2}, -r a_n e_{n-1}, r b_n e_{n-1}, -b_{n-1} e_n, a_{n-1} e_n] .
	\]
	Using \eqref{eq: induction P^n_o SL(n)-contravariant - mu_n} and \eqref{eq: induction P^n_o SL(n)-contravariant - mu_*},
	we can compare degrees of homogeneity in $r > 0$ to see that $2 k_2 = k_3$ and $k_0 = k_1 = \tilde k_3 = 0$.
	Therefore, the vector space of measurable valuations on $\cQ^n$ that are $\SLn$-contravariant is at most $1$-dimensional.
	Since the map $P \mapsto m (P^*)$ is a measurable valuation on $\cQ^n$ which is $\SLn$-contravariant,
	we must have $\mu (P) = k m (P^*) $ for some constant $k \in \R$ and all polytopes $P \in \cQ^n$.
	An application of Lemma \ref{le: induction R^n_o SL(n)-contravariant vanishes} to the difference $\mu (P) - k m (P^*)$
	and a glance at Theorem \ref{th: determined by values on SL(n)(R^n)} complete the proof.
\end{proof}

Next, we establish two facts on matrix valued valuations which will be crucial for our induction.

\begin{lemma}\label{le: R^n_o GL(n)-covariant matrix-valued vanishes}
	Let $n \geq 3$.
	Assume that $\mu \colon \cR^n \to \R^{n \times n}$ is a measurable valuation which is $\GLn$-covariant and vanishes on $\cQ^n$.
	Furthermore, assume that $\mu (P)$ is an antisymmetric matrix for all $P \in \cP^n_o$.
	Then $\mu$ vanishes on $\cR^n$.
\end{lemma}
\begin{proof}
	Suppose that $B, c, d, x, y$ form a double pyramid.
	By the valuation property
	\begin{multline*}
		\mu \left[ B, -c \ptwo x 1, d \ptwo y 1 \right] + \mu \left[ B, -s \ptwo y 1, t \ptwo y 1 \right] = \\
		\mu \left[ B, -c \ptwo x 1, t \ptwo y 1 \right] + \mu \left[ B, -s \ptwo y 1, d \ptwo y 1 \right]
	\end{multline*}
	for sufficiently small $s,t > 0$.
	The $\GLn$-covariance and the assumption that $\mu$ vanishes on $\cQ^n$ yield
	\[
		\mu \left[ B, -s \ptwo y 1, t \ptwo y 1 \right] = 0
		\quad \textrm{ and } \quad
		\mu \left[ B, -s \ptwo y 1, d \ptwo y 1 \right] = 0 .
	\]
	Therefore,
	\[
		\mu \left[ B, -c \ptwo x 1, d \ptwo y 1 \right] = \mu \left[ B, -c \ptwo x 1, t \ptwo y 1 \right] ,
	\]
	i.e.\ the left hand side is independent of $d$.
	Similarly, we see that it is also independent of $c$.
	So for each $B \in \cP^{n-1}$ we can define a function $f^B \colon \R^{n-1} \times \R^{n-1} \to \R^{n \times n}$ by
	\[
		f^B(x,y) = \mu \left[ B, -c \ptwo x 1, d \ptwo y 1 \right] ,
	\]
	as long as $B, c, d, x, y$ form a double pyramid.
	By the $\GLn$-covariance we have
	\begin{equation} \label{eq: f^B(x-y, 0) matrix-valued}
		f^B(x, y) = \ptwotwo {\Id'} 0 y 1 f^B(x-y, 0) \ptwotwo {\Id'} {y^t} 0 1 .
	\end{equation}
	The valuation property again implies that
	\[
		\mu \left[ B, -c \ptwo x 1, d \ptwo y 1 \right] + \mu [ B, -r e_n, r e_n] =
		\mu \left[ B, -c \ptwo x 1, r e_n \right] + \mu \left[ B, -r e_n, d \ptwo y 1 \right]
	\]
	for sufficiently small $r > 0$.
	Since $\mu [ B, -r e_n, r e_n] = 0$ and by using \eqref{eq: f^B(x-y, 0) matrix-valued}, this yields
	\begin{equation} \label{eq: f^B(x, 0) matrix-valued}
		f^B(x+y, 0) = \ptwotwo {\Id'} 0 y 1 f^B(x, 0) \ptwotwo {\Id'} {y^t} 0 1 + f^B(y, 0) .
	\end{equation}
	By the antisymmetry, $f^B_{nn}$ vanishes.
	Using \eqref{eq: f^B(x, 0) matrix-valued}, it is easy to see that $x \mapsto f^B_{* n}(x,0)$ satisfies
	Cauchy's functional equation.
	The measurability of $\mu$ implies that also $f^B_{* n}$ is measurable.
	Hence $x \mapsto f^B_{* n}(x,0)$ is linear.
	The definition of $f^B$ therefore implies that
	\begin{equation}\label{*n linearity}
	  \mu_{*n} \left[ I, -c \ptwo {rx} 1, d e_n \right] = r \mu_{*n} \left[ I, -c \ptwo x 1, d e_n \right]
	\end{equation}
	for $r > 0$.
	On the one hand, by the $\GLn$-covariance and \eqref{*n linearity}, we obtain
	\begin{align*}
		\mu_{*n} \left[ r B, -c \ptwo x 1, d e_n \right]
		&= \mu_{*n} \left( \ptwotwo {r \Id'} 0 0 1 \left[ B, -c \ptwo {\frac x r} 1, d e_n \right] \right) \\
		&= r^n \mu_{*n} \left[ B, -c \ptwo {\frac x r} 1, d e_n \right] \\
		&= r^{n-1} \mu_{*n} \left[ B, -c \ptwo x 1, d e_n \right] .
	\end{align*}
	On the other hand, the $\GLn$-covariance and the independence with respect to $c$ and $d$ prove
	\[
		\mu_{*n} \left[ r B, -c \ptwo x 1, d e_n \right]
		= \mu_{*n} \left( r \left[ B, - \frac c r \ptwo x 1, \frac d r e_n \right] \right)
		= r^{n+2} \mu_{*n} \left[ B, -c \ptwo x 1, d e_n \right] .
	\]
	Therefore, using \eqref{eq: f^B(x-y, 0) matrix-valued}, $\mu_{*n}$ and by antisymmetry also $\mu_{n*}$ vanish on $\cR^n$.
	Using what we have just shown, the same approach yields that
	\[
		r^n \mu_{**} \left[ B, -c \ptwo x 1, d e_n \right] = r^{n+2} \mu_{**} \left[ B, -c \ptwo x 1, d e_n \right] .
	\]
	Consequently, $\mu_{**}$ vanishes on $\cR^n$.
\end{proof}

\begin{lemma}\label{le: induction P^n_o GL(n)-covariant matrix-valued}
	Let $n \geq 3$ and suppose that Theorems \ref{th: P^n_o SL(n)-covariant}
	and \ref{th: P^n_o GL(n)-covariant matrix-valued} hold in dimension $n-1$.
	Assume that $\mu \colon \cP^n_o \to \R^{n \times n}$ is a measurable valuation which is $\GLn$-covariant.
	Then there exists a $k \in \R$ such that
	\[
		\mu (P) = k \, M_2(P)
	\]
	for all $P \in \cP^n_o$.
\end{lemma}
\begin{proof}
	As in the proof of Theorem \ref{th: P^2_o GL(2)-covariant matrix-valued},
	it suffices to show that $\mu$ vanishes
	if $\mu (P)$ is antisymmetric for all $P \in \cP^n_o$.
	Moreover, by Lemma \ref{le: R^n_o GL(n)-covariant matrix-valued vanishes}, the assumed $\GLn$-covariance
	and Theorem \ref{th: determined by values on SL(n)(R^n)},
	it is enough to show that $\mu$ vanishes on $\cQ^n$.
	
	For fixed $J$, the map $B \mapsto \mu_{**} [B, J]$ satisfies the conditions of Theorem
	\ref{th: P^n_o GL(n)-covariant matrix-valued} in dimension $n-1$.
	In particular, it has symmetric images.
	But, by assumption, these images are also antisymmetric.
	So the map $B \mapsto \mu_{**} [B, J]$ vanishes on $\cP^{n-1}_o$.
	
	By the antisymmetry, the only thing left to show is that $\mu_{*n}$ vanishes.
	The expression $\mu_{*n} [B, J]$ is $\SLn[n-1]$-covariant and $n$-homogeneous in $B$,
	and odd and $2$-homogeneous in $J$.
	Since it is also a measurable valuation in both arguments,
	we deduce from Theorem \ref{th: dim 1 odd} and the assumption that Theorem \ref{th: P^n_o SL(n)-covariant}
	holds in dimension $n-1$,
	\[
		\mu_{*n} [B, J] = k \left( d^2 - c^2 \right) m'(B) .
	\]
	For $i \in \{ 1, \ldots, n-2 \}$, let $I_i$ be a line segment in $\spn \{ e_i \}$ containing the origin in its interior.
	The $\GLn$-covariance of $\mu$ applied to the map which interchanges the $(n-1)$-st and $n$-th component of a vector yields
	\[
		\mu_{n-1, n} \left[ I_1, \ldots, I_{n-2}, -c e_{n-1}, d e_{n-1}, J \right] =
		\mu_{n, n-1} \left[ I_1, \ldots, I_{n-2}, -c e_{n-1}, d e_{n-1}, J \right] .
	\]
	By the antisymmetry, the above quantities vanish.
	Therefore, $k = 0$.
	This completes the proof.
\end{proof}

After these preparations we can now prove the main theorems of this section by induction on the dimension.
\begin{proof}
	[
		Proof of Theorems
		\ref{th: P^n_o SL(n)-covariant}, \ref{th: P^n_o GL(n)-covariant matrix-valued},
		\ref{th: P^n_o SL(n)-contravariant} and \ref{th: P^n_o GL(n)-contravariant matrix-valued}
	]
	For $n=2$, these theorems have been proved in the previous section.
	Assume that all four theorems hold in dimension $n-1$.
	By Lemma \ref{le: induction P^n_o SL(n)-contravariant}, Theorem \ref{th: P^n_o SL(n)-contravariant} holds in dimension $n$.
	Moreover, Lemma \ref{le: induction P^n_o GL(n)-covariant matrix-valued} shows that
	Theorem \ref{th: P^n_o GL(n)-covariant matrix-valued} also holds in dimension $n$.
	Recall that $P \mapsto P^*$ transforms covariance into contravariance and vice versa.
	Hence, for dimension $n$, Theorems \ref{th: P^n_o SL(n)-covariant} and \ref{th: P^n_o GL(n)-contravariant matrix-valued}
	follow directly from Theorems \ref{th: P^n_o SL(n)-contravariant} and \ref{th: P^n_o GL(n)-covariant matrix-valued}, respectively.
\end{proof}

%% file: acknowledgements.tex
The work of the second author was supported by the European Research Council (ERC)
within the project ``Isoperimetric inequalities and integral geometry'', Project number: 306445.
The work of the second author was also supported by the ETH Zurich Postdoctoral Fellowship Program
and the Marie Curie Actions for People COFUND Program.
The second author wants to especially thank Michael Eichmair for welcoming him at the ETH Zurich and being his mentor there.